\documentclass[11pt]{article}
\usepackage[margin=1in]{geometry}
\usepackage[utf8]{inputenc}
\usepackage[T1]{fontenc}
\usepackage[english]{babel}
\usepackage{lmodern}
\usepackage{graphicx}
\usepackage{wrapfig}
\usepackage{subfig}
\usepackage{caption}
\usepackage{paralist}
\usepackage{enumitem}

\usepackage[linesnumbered,ruled,vlined]{algorithm2e}

\graphicspath{{./figures/}}

\makeatletter
\renewcommand{\p@enumii}{}
\makeatother

\newcommand{\titel}{
	Hamiltonian cycles in 4-connected planar and projective planar triangulations with few 4-separators
	}

\usepackage[usenames]{color}
\definecolor{hellblau}{rgb}{0.2,0.4,1}
\definecolor{dunkelblau}{rgb}{0,0,0.8}
\definecolor{dunkelgruen}{rgb}{0,0.5,0}
\usepackage[
	pdftex,
	colorlinks,
	linkcolor=dunkelblau,
	urlcolor=dunkelblau,
	citecolor=dunkelgruen,
	bookmarks=true,
	linktocpage=true,
	pdfsubject={}
]{hyperref}
\usepackage{pdfpages}
\usepackage{amsmath}
\usepackage{amsthm}
\usepackage{amsfonts}
\theoremstyle{plain}
	\newtheorem{satz}{Satz}[]
	\newtheorem{theorem}[satz]{Theorem}
	\newtheorem{lemma}[satz]{Lemma}

\theoremstyle{remark}
	
\theoremstyle{definition}

	\newtheorem{conjecture}[satz]{Conjecture}

\begin{document}
	\title{\titel}
	\author{On-Hei Solomon Lo\thanks{School of Mathematical Sciences, Xiamen University, Xiamen 361005, PR China. This work was partially supported by NSFC grant 11971406.}\\
	\and Jianguo Qian\addtocounter{footnote}{-1}\footnotemark}
	\date{}
	\maketitle

\begin{abstract}
	Whitney proved in 1931 that every 4-connected planar triangulation is hamiltonian. Later in 1979, {Hakimi, Schmeichel and Thomassen conjectured that every such triangulation on $n$ vertices has at least $2(n - 2)(n - 4)$ hamiltonian cycles. Along this direction, Brinkmann, Souffriau and Van Cleemput established} a linear lower bound on the number of hamiltonian cycles in 4-connected planar triangulations. {In stark contrast, Alahmadi, Aldred and Thomassen showed that every 5-connected triangulation of the plane or the projective plane has exponentially many hamiltonian cycles. This gives the motivation to study the number of hamiltonian cycles of 4-connected triangulations with few 4-separators. Recently,} Liu and Yu showed that every 4-connected planar triangulation with $O(n / \log n)$ 4-separators has a quadratic number of hamiltonian cycles. {By adapting the framework of Alahmadi et al.\ we strengthen the last two aforementioned results. We prove that every 4-connected planar or projective planar triangulation with $O(n)$ 4-separators has exponentially many hamiltonian cycles.}
\end{abstract}

\section{Introduction}

A classical theorem of Whitney in 1931 proved that every 4-connected planar triangulation has a hamiltonian cycle~\cite{Whitney1931}. In 1956, this result was extended by Tutte~\cite{Tutte1956} to 4-connected planar graphs. One may subsequently ask how many hamiltonian cycles a 4-connected planar triangulation or planar graph may have. In 1979, Hakimi, Schmeichel and Thomassen~\cite{Hakimi1979a} proposed the following conjecture:

\begin{conjecture}[\cite{Hakimi1979a}] \label{conj:HST}
	Every $4$-connected planar triangulation $G$ on $n$ vertices has at least $2(n - 2)(n - 4)$ hamiltonian cycles, with equality if and only if {$G$ is the double-wheel graph on $n$ vertices, that is, the join of a cycle of length $n - 2$ and an empty graph on two vertices.}
\end{conjecture}

In the same paper, they also proved that every 4-connected planar triangulation on $n$ vertices has at least $n / \log_2 n$ hamiltonian cycles. This lower bound was recently improved by Brinkmann, Souffriau and Van Cleemput~\cite{Brinkmann2018} to a linear bound of {$12(n - 2)/5$}{, which was then refined to }{$161(n - 2)/60$} for $n \ge 7$ by Cuvelier~\cite{Cuvelier2015}. For 4-connected planar graphs, Sander~\cite{Sanders1997} showed that there exists a hamiltonian cycle containing any two prescribed edges in any 4-connected planar graph. This implies that every 4-connected planar graph $G$ has at least $\binom{\Delta}{2}$ hamiltonian cycles, where $\Delta$ denotes the maximum degree of $G$. However, it assures only a constant lower bound as there are infinitely many 4-connected planar graphs of {maximum degree upper} bounded by 4. Generalizing the method used in~\cite{Brinkmann2018}, Brinkmann and Van Cleemput~\cite{Brinkmann} gave a linear lower bound on the number of hamiltonian cycles in 4-connected planar graphs.

We will focus on the number of hamiltonian cycles of 4-connected triangulations of the plane or the projective plane with a bounded number of 4-separators. Interestingly, if Conjecture~\ref{conj:HST} held, we would have that a 4-connected planar triangulation has a minimum number of hamiltonian cycles if and only if it has a maximum number of 4-separators, as the double-wheel graphs are the 4-connected planar triangulations that maximize the number of 4-separators~\cite{Hakimi1979}. {Trivially, 5-connected planar triangulations have a minimum number of 4-separators among 4-connected planar triangulations as they have no 4-separators. Indeed, Alahmadi, Aldred and Thomassen~\cite{Alahmadi2020} proved that every 5-connected triangulation embedded on the plane or on the projective plane has $2^{\Omega(n)}$ hamiltonian cycles. Following the approach of Alahmadi et al.,} it was shown in~\cite{Lo2020a} that every 4-connected planar triangulation has at least $\Omega((n / \log n)^2)$ hamiltonian cycles if it has only $O(\log n)$ 4-separators. Recently, Liu and Yu~\cite{Liu2020} improved this result by showing that every 4-connected planar triangulation with $O(n / \log n)$ 4-separators has $\Omega(n^2)$ hamiltonian cycles. 
These results, in a sense, give evidence supporting that triangulations of the plane or the projective plane with fewer 4-separators may have more hamiltonian cycles.
We will prove the following result, indicating that a 4-connected planar or projective planar triangulation may have exponentially many hamiltonian cycles as long as it has at most a linear number of $4$-separators, thereby extending the results mentioned above.

\begin{theorem} \label{thm:main}
	Let $G$ be a $4$-connected planar or projective planar triangulation on $n$ vertices and let $c$ be an arbitrary constant less than $1/324$. If $G$ has at most $cn$ $4$-separators, then it has $2^{\Omega(n)}$ hamiltonian cycles.
\end{theorem}

\section{Results}

In this section we first prepare some lemmas which will be used to construct a vertex subset with several properties (see also Lemma~\ref{lem:main}). The proof of Theorem~\ref{thm:main} will be given at the end of this section.

For notation and terminology not explicitly defined in this paper, we refer the reader to~\cite{Bondy2008, Mohar2001}. 
{A vertex subset or a subgraph of a connected graph is \emph{separating} if its removal disconnects the graph. We call a separating vertex set on $k$ vertices a \emph{$k$-separator}. A {\it $k$-cycle} is a cycle of length $k$.} 
Let $S$ be an independent set. We say $S$ \emph{saturates} a 4- or 5-cycle $C$ if $S$ contains two vertices of $C$. A \emph{diamond-$6$-cycle} is the graph depicted in Figure~\ref{fig:diamond6}, where the white vertices are called \emph{crucial}. We say $S$ \emph{saturates} a diamond-6-cycle $D$ if $S$ contains three crucial vertices of $D$. {Recall that the \emph{Euler genus} $eg(\Sigma)$ of a surface $\Sigma$ is defined to be $2 - \chi(\Sigma)$, where $\chi(\Sigma)$ denotes the Euler characteristic of $\Sigma$. A graph $H$ is {\it $d$-degenerate} if every induced subgraph of $H$ has a vertex of degree at most $d$. It is well known that every $d$-degenerate graph $H$ is $(d+1)$-colorable and hence has an independent set of at least $|V(H)|/(d + 1)$ vertices.}

\begin{figure}[h!t]
	\centering
	\includegraphics[height=2.8cm]{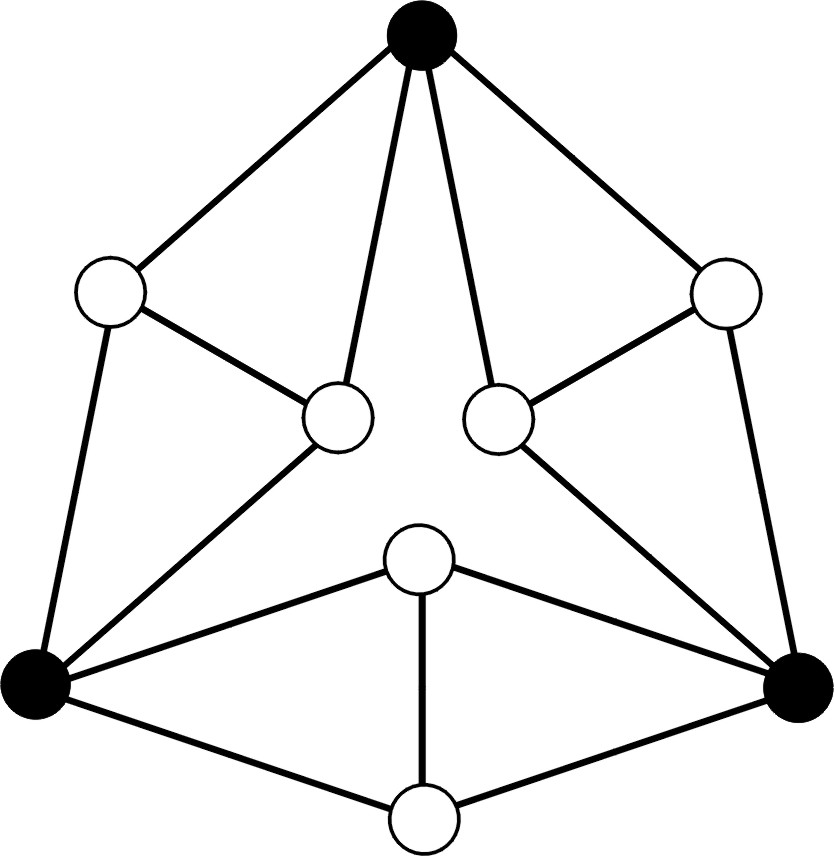}
	\caption{A diamond-6-cycle with six crucial vertices (white).}
	\label{fig:diamond6}
\end{figure}

The following tool is due to Alahmadi et al.~\cite{Alahmadi2020}, which helps finding homotopic curves from a sufficiently large family of curves.

\begin{lemma}[{\cite[Corollary~2]{Alahmadi2020}}] \label{lem:homotopic}
	Let $\Sigma$ be a surface of Euler genus $\sigma$ and let $\mathcal{C}$ be a family of simple closed curves on $\Sigma$ with the property that every $C \in \mathcal{C}$ can have at most one point that is contained in other curves in $\mathcal{C}$. Let $r$ be a positive integer. If $|\mathcal{C}| \ge 5(r - 1)\sigma + 1$, then there are $r$ homotopic curves in $\mathcal{C}$.
\end{lemma}

The number of the vertices adjacent to at least three vertices of a 4-cycle on a surface can be bounded as follows.

\begin{lemma} \label{lem:3on4cycle}
	Let $G$ be a triangulation of a surface of Euler genus $\sigma$, and $C$ be a $4$-cycle in $G$. Denote $V_C := \{v \in V(G) : |N_G(v) \cap V(C)| \ge 3\}$. Then $|V_C| \le 8(\sigma + 1)$.
\end{lemma}
\begin{proof}
	Suppose to the contrary that $|V_C| > 8(\sigma + 1)$. We have that $G$ contains $K_{3, q}$ as a subgraph, where $q = 2(\sigma + 1) + 1$. This is however impossible since, by a theorem of Ringel~\cite{Ringel1965, Ringel1965a}, the Euler genus of $K_{3,q}$ equals $eg(K_{3, q}) = \lceil \frac{q - 2}{2} \rceil = \sigma + 1>\sigma$.
\end{proof}

The next three lemmas aim at finding a vertex set that saturates no 4-cycle, or 5-cycle, or diamond-6-cycle. {The main idea of the proofs comes from~\cite{Alahmadi2020}.}

\begin{lemma} \label{lem:cleanC4}
	Let $G$ be a triangulation of a surface of Euler genus $\sigma$ and let $S \subseteq V(G)$ be an independent set of vertices of degree at most $6$. If  $S$ saturates no separating $4$-cycle in $G$, then $S$ has a subset of size at least $|S|/c$ that saturates no $4$-cycle, where $c := \binom{6}{2} (10\sigma + 1) + 1$.
\end{lemma}
\begin{proof}
	Let $H$ be the graph on the vertex set $S$ such that two vertices are adjacent if they saturate some 4-cycle in $G$. It suffices to show that $H$ has maximum degree at most $d := c - 1 = \binom62 (10\sigma + 1)$ and hence chromatic number at most $c$. Suppose, to the contrary, there exists $v \in S = V(H)$ with $d_H(v) > d$. {As $d_G(v)\leq 6$, there are $u, w \in N_G(v)$ such that the path $uvw$ are contained in at least $\lceil (d+1)/\binom{6}{2} \rceil = (10\sigma + 1) + 1$ 4-cycles in $G$.} We may contract the path $u v w$ and apply Lemma~\ref{lem:homotopic} to show that there are at least three homotopic 4-cycles containing the path $u v w$ and saturated by $S$. This yields a separating 4-cycle saturated by $S$ and hence contradicts our assumption. Thus the lemma follows.
\end{proof}
\begin{lemma} \label{lem:cleanC5}
	Let $G$ be a triangulation of a surface of Euler genus $\sigma$ and let $S \subseteq V(G)$ be an independent set of vertices of degree at most $6$. If $S$ saturates no $4$-cycle in $G$, then $S$ has a subset of size at least $|S|/c$ that saturates no $5$-cycle, where $c := 2 \binom{6}{2} (40\sigma + 1) + 1$.
\end{lemma}
\begin{proof} 
	Let $H$ be the graph on the vertex set $S$ in which two vertices are adjacent if they saturate some 5-cycle in $G$. Let $d := c - 1 = 2 \binom62 (40\sigma + 1)$.  It suffices to show that $H$ is $d$-degenerate. Let $K$ be any induced subgraph of $H$. We will show that $K$ has a vertex of degree at most $d$.

	We first consider the following. Let $v$ be a vertex in $K$ with $d_K(v) \ge d$. Since $d_G(v) \le 6$, there exist $u, w \in N_G(v)$, distinct vertices $x_1, \dots, x_{2t} \in N_K(v)$ and $y_1, \dots, y_{2t} \in V(G) \setminus S$, where {$t := 40\sigma + 1$}, such that for every $1 \le i \le 2t$ either $u v w x_i y_i u$ or $u v w y_i x_i u$ is a 5-cycle saturated by $V(K)$. Denote by $C_i$ the 5-cycle saturated by $v$ and $x_i$,  and denote by $P_i$ the path obtained from $C_i$ by deleting $v$ ($1 \le i \le 2t$). We claim that at least $t$ vertices of $y_1, \dots, y_{2t}$ are distinct. Otherwise there are $1 \le i < j < k \le 2t$ such that $y_i = y_j = y_k$. Then $u$ or $w$ is adjacent to two of $x_i,x_j,x_k$, say $x_i, x_j \in N_G(w)$. However, this implies that $w x_i y_i x_j w$ is a 4-cycle saturated by $V(K)$, which contradicts our assumption. Therefore, we may assume that the paths $P_1, \dots, P_t$ are pairwise internally disjoint.
	{By contracting the path $uvw$ and applying Lemma~\ref{lem:homotopic}, we may obtain nine homotopic curves from $P_1, \dots, P_t$, say $P_1, \dots, P_9$. Relabelling if necessary, we may further assume that the closed disc $D_v$ bounded by $P_1 \cup P_9$ contains $P_1, \dots, P_9$, and the closed disc bounded by $P_1 \cup P_5$ contains $P_1, \dots, P_5$ but not $P_6, \dots, P_9$.}

	{We now show that $K$ has minimum degree at most $d$. Suppose to the contrary that $d_K(v) > d$ for every $v \in V(K)$. We choose a vertex $v \in V(K)$ and the associated nine homotopic curves such that the number of vertices of $G$ contained in $D_v$ is minimum.
	
	Let $C$ be a 5-cycle in $G$ containing $x_5$ and another vertex $v' \in V(K)$ outside of $D_v$. We claim that $v' = v$. Let $P$ be the minimal path in $C$ containing $x_5$ with end-vertices in $P_1 \cup P_9$. Notice that $C$ has length five and $x_5$ lies in the interior but not the boundary of the disc $D_v$.} So by the arrangement of $P_1, \dots, P_9$, the end-vertices of $P$ must be $u$ and $w$. Since $V(K)$ is an independent set of $G$ and $v, x_5 \in V(K)$, we have $v' \notin N_G(x_5) \cup \{u, w\}$. Therefore $P$ has length less than four. If $P$ has length two, then $u v w x_5 u$ would be a 4-cycle saturated by $V(K) \subseteq S$, which contradicts our assumption. If $P$ has length three, then we have $v' = v$ (otherwise $u v w v' u$ would be a 4-cycle saturated by $V(K)$). This thus establishes our claim.

	 Our previous claim implies that all 5-cycles in $G$ containing $x_5$ and another vertex of $V(K)$ other than $v$ must lie in $D_v$. 
	 As $d_{K - v}(x_5) = d_K(x_5) - 1 \ge d$, we may apply our discussion above to $K - v$ and $x_5$ (instead of $K$ and $v$) to obtain nine homotopic curves associated with $x_5$ and a closed disc containing them. Then we may have $D_{x_5} \subseteq D_v$.   Moreover, it is not hard to see that $D_{x_5}$ does not contain $x_1$. Therefore the number of vertices in $D_{x_5}$ is strictly less than that of $D_v$. This contradicts  our choice of $v$ and $D_v$ and hence the result follows.
\end{proof}
	
The proof of following lemma is omitted as it can be readily deduced from the proof of~\cite[Lemma~10]{Alahmadi2020}.

\begin{lemma}[{\cite[Lemma~10]{Alahmadi2020}}] \label{lem:cleanC6}
	Let $G$ be a triangulation of a surface of Euler genus $\sigma$ and let $S \subseteq V(G)$ be an independent set of vertices of degree at most $6$. If $S$ saturates no $4$-cycle in $G$, then $S$ has a subset of size at least $|S|/c$ that saturates no diamond-$6$-cycle, where $c$ is a positive constant depending only on $\sigma$.
\end{lemma}

One of the key ingredients of the proof of exponential lower bound on the number of hamiltonian cycles in 5-connected triangulations given by Alahmadi et al.~\cite{Alahmadi2020} is to find {many edge sets $F \subseteq E(G)$} so that $G - F$ is 4-connected. Their approach was refined by~\cite{Lo2020a, Liu2020} for 4-connected planar triangulations. In order to prove our result for triangulations of the projective plane, we need to generalize a lemma given by Liu and Yu~\cite{Liu2020} to surfaces of higher genus.

Let $G$ be a triangulation of any surface and $A \subseteq V(G)$ be a 3-separator of $G$. It is shown in the proof of~\cite[Lemma~1]{Alahmadi2020} and its subsequent discussion that $G[A]$ is a surface separating 3-cycle. Using this fact and a theorem of Thomas and Yu~\cite{Thomas1994} that every 4-connected projective planar graph is hamiltonian, the proof of~\cite[Lemma~2.1]{Liu2020} can be easily modified to show the following result. We omit the proof.

\begin{lemma}[{\cite[Lemma~2.1]{Liu2020}}] \label{lem:main}
	Let $G$ be a $4$-connected triangulation of a surface $\Sigma$ of Euler genus $\sigma$. Let $S \subseteq V(G)$ be a vertex subset satisfying the following conditions: \begin{enumerate}
		\item[(i)] $d_G(v) \le 6$ for any $v \in S$;
		\item[(ii)] $S$ is an independent vertex set;
		\item[(iii)] no vertex in $S$ is contained in any separating $4$-cycle in $G$;
		\item[(iv)] no vertex in $S$ is adjacent to three vertices of any separating $4$-cycle in $G$; and
		\item[(v)] $S$ saturates no $4$-, $5$- or diamond-$6$-cycle.
	\end{enumerate}
	Let $F \subseteq E(G)$ be any edge subset such that $|F| = |S|$ and for any $v \in S$ there is precisely one edge in $F$ incident with $v$. Then $G - F$ is $4$-connected. Moreover, if $\Sigma$ is the plane or the projective plane, then $G$ has $2^{\Omega(|S|)}$ hamiltonian cycles.
\end{lemma}

We are now ready to prove Theorem~\ref{thm:main}.

\begin{proof} [Proof of Theorem~\ref{thm:main}] 
	
	To prove the theorem, it suffices to construct a vertex set $S \subseteq V(G)$ of size $\Omega(n)$ satisfying conditions~(i) to~(v) in Lemma~\ref{lem:main}.
	
	{Since $G$ has average degree less than 6 and minimum degree at least 4, the set $S_1$ of vertices of degree at most 6 has size at least $n/3$. We may subsequently obtain a vertex set $S_2 \subseteq S_1$ of size at least $|S_1| / 6 \ge n/18$ satisfying~(i) and~(ii) of Lemma~\ref{lem:main}, as planar and projective planar graphs are 6-colorable.}

	Let $C$ be any separating 4-cycle in $G$. Since $S_2$ is an independent set, {it follows from} Lemma~\ref{lem:3on4cycle} that $S_2$ has at most $16+2=18$ vertices that are contained in $C$ or adjacent to three vertices of $C$. Deleting these vertices from $S_2$ for every separating 4-cycle, we may obtain a vertex set $S_3 \subseteq S_2$ satisfying~(i) to~(iv) of Lemma~\ref{lem:main} and $|S_3|\geq n/18-18cn=(1/18-324c/18)n$. Recall that $c$ is a constant less than $1/324$. This means that $1/18-324c/18$ is a positive constant and hence $|S_3|=\Omega(n)$.
	
	As no vertex in $S_3$ is contained in any separating 4-cycle, no 4-cycle saturated by $S_3$ is separating. Successively applying Lemmas~\ref{lem:cleanC4},~\ref{lem:cleanC5} and~\ref{lem:cleanC6}, we obtain a vertex set $S \subseteq S_3$ of size $\Omega(n)$ satisfying (i) to (v) of Lemma~\ref{lem:main}, implying that  $G$ has $2^{\Omega(n)}$ hamiltonian cycles. This completes our proof.
\end{proof}

We remark that Gr\"unbaum~\cite{Gruenbaum1970} and Nash-Williams~\cite{Nash-Williams1973}
independently conjectured that every 4-connected toroidal graph is hamiltonian. The truth of this conjecture (respectively, an analogue for the Klein bottle) would extend Theorem~\ref{thm:main} to triangulations of the torus (respectively, the Klein bottle). {Note that there are non-hamiltonian 4-connected graphs that are embedded in the double torus or in the surface obtained from the sphere by attaching three crosscaps (see~\cite{Kawarabayashi2016}).}

\bibliographystyle{abbrv}
\bibliography{paper}

\begin{thebibliography}{10}

\bibitem{Alahmadi2020}
A.~Alahmadi, R.~Aldred, and C.~Thomassen.
\newblock Cycles in 5-connected triangulations.
\newblock {\em J. Combin. Theory Ser. B}, 140:27--44, 2020.

\bibitem{Bondy2008}
J.~A. Bondy and U.~S.~R. Murty.
\newblock {\em Graph {T}heory}, volume 244 of {\em Graduate Texts in
  Mathematics}.
\newblock Springer, New York, 2008.

\bibitem{Brinkmann2018}
G.~Brinkmann, J.~Souffriau, and N.~{Van Cleemput}.
\newblock On the number of hamiltonian cycles in triangulations with few
  separating triangles.
\newblock {\em J. Graph Theory}, 87(2):164--175, 2018.

\bibitem{Brinkmann}
G.~Brinkmann and N.~{Van Cleemput}.
\newblock 4-connected polyhedra have at least a linear number of hamiltonian
  cycles.
\newblock Manuscript.

\bibitem{Nash-Williams1973}
{C. St. J. A. Nash-Williams}.
\newblock Unexplored and semi-explored territories in graph theory.
\newblock In {\em New directions in the theory of graphs}, pages 149--186.
  Academic Press, New York, 1973.

\bibitem{Cuvelier2015}
A.~Cuvelier.
\newblock Grenzen voor het aantal {H}amiltoniaanse cykels in triangulaties.
\newblock Master's thesis, Universiteit Gent, 2015.

\bibitem{Gruenbaum1970}
B.~Grünbaum.
\newblock Polytopes, graphs, and complexes.
\newblock {\em Bull. Amer. Math. Soc.}, 76:1131--1201, 1970.

\bibitem{Hakimi1979}
S.~L. Hakimi and E.~F. Schmeichel.
\newblock On the number of cycles of length $k$ in a maximal planar graph.
\newblock {\em J. Graph Theory}, 3(1):69--86, 1979.

\bibitem{Hakimi1979a}
S.~L. Hakimi, E.~F. Schmeichel, and C.~Thomassen.
\newblock On the number of hamiltonian cycles in a maximal planar graph.
\newblock {\em J. Graph Theory}, 3(4):365--370, 1979.

\bibitem{Kawarabayashi2016}
K.~Kawarabayashi and K.~Ozeki.
\newblock 5-{C}onnected {T}oroidal {G}raphs are {H}amiltonian-{C}onnected.
\newblock {\em SIAM J. Discrete Math.}, 30(1):112--140, 2016.

\bibitem{Liu2020}
X.~Liu and X.~Yu.
\newblock Number of {H}amiltonian cycles in planar triangulations, 2020.
\newblock To appear in \emph{SIAM J. Discrete Math.}.

\bibitem{Lo2020a}
O.-H.~S. Lo.
\newblock Hamiltonian cycles in 4-connected plane triangulations with few
  4-separators.
\newblock {\em Discrete Math.}, 343(12):112126, 2020.

\bibitem{Mohar2001}
B.~Mohar and C.~Thomassen.
\newblock {\em Graphs on {S}urfaces}.
\newblock Johns Hopkins Studies in the Mathematical Sciences. Johns Hopkins
  University Press, Baltimore, 2001.

\bibitem{Ringel1965}
G.~Ringel.
\newblock Das {G}eschlecht des vollst\"andigen paaren {G}raphen.
\newblock {\em Abh. Math. Sem. Univ. Hamburg}, 28:139--150, 1965.

\bibitem{Ringel1965a}
G.~Ringel.
\newblock Der vollst\"andige paare {G}raph auf nichtorientierbaren {F}l\"achen.
\newblock {\em J. Reine Angew. Math.}, 220:88--93, 1965.

\bibitem{Sanders1997}
D.~P. Sanders.
\newblock On paths in planar graphs.
\newblock {\em J. Graph Theory}, 24:341–345, 1997.

\bibitem{Thomas1994}
R.~Thomas and X.~Yu.
\newblock 4-connected projective-planar graphs are hamiltonian.
\newblock {\em J. Combin. Theory Ser. B}, 62:114–132, 1994.

\bibitem{Tutte1956}
W.~T. Tutte.
\newblock A theorem on planar graphs.
\newblock {\em Trans. Amer. Math. Soc.}, 82:99–116, 1956.

\bibitem{Whitney1931}
H.~Whitney.
\newblock A theorem on graphs.
\newblock {\em Ann. of Math.}, 32(2):378--390, 1931.

\end{thebibliography}

\end{document}